\documentclass{amsart}
\usepackage{amssymb}
\usepackage{amsfonts}

\newtheorem{theorem}{Theorem}
\theoremstyle{plain}

\newtheorem{corollary}{Corollary}

\newtheorem{definition}{Definition}

\newtheorem{lemma}{Lemma}

\newtheorem{proposition}{Proposition}

\numberwithin{equation}{section}
\input{tcilatex}

\begin{document}
\title[Inequalities for Godunova-Levin Class Functions]{Some integral
inequalities for Godunova-Levin Class Functions}
\author{M.Emin Ozdemir$^{\blacklozenge }$}
\address{$^{\blacklozenge }$ATATURK UNIVERSITY, K.K. EDUCATION FACULTY,
DEPARTMENT OF MATHEMATICS, 25240 CAMPUS, ERZURUM, TURKEY}
\email{emos@atauni.edu.tr}
\author{Merve Avci$^{\ast ,\diamondsuit }$}
\thanks{$^{\diamondsuit }$Corresponding Author}
\address{$^{\ast }$ADIYAMAN UNIVERSITY, FACULTY OF SCIENCE AND ART,
DEPARTMENT OF MATHEMATICS, 02040, ADIYAMAN, TURKEY}
\email{mavci@posta.adiyama.edu.tr}

\begin{abstract}
In this paper, we obtain some new inequalities for functions which are
introduced by Godunova and Levin.
\end{abstract}

\maketitle

\section{introduction}

Following inequalities are well known in the literature as Hermite-Hadamard
inequality and Simpson inequality respectively:

\begin{theorem}
\label{teo 1.1} Let $f:I\subseteq 
\mathbb{R}
\rightarrow 
\mathbb{R}
$be a convex function on the interval $I$ of real numbers and $a,b\in I$
with $a<b.$Then, the following double inequality holds 
\begin{equation*}
f\left( \frac{a+b}{2}\right) \leq \frac{1}{b-a}\int_{a}^{b}f(x)dx\leq \frac{%
f(a)+f(b)}{2}.
\end{equation*}
\end{theorem}

\begin{theorem}
\label{teo 1.2} Let $f:\left[ a,b\right] \rightarrow 
\mathbb{R}
$ be a four times continuously differentiable mapping on $\left( a,b\right) $
and $\left\Vert f^{(4)}\right\Vert _{\infty }=\sup\limits_{x\in \left(
a,b\right) }\left\vert f^{(4)}\left( x\right) \right\vert <\infty .$ Then,
the following inequality holds:%
\begin{equation*}
\left\vert \frac{1}{3}\left[ \frac{f(a)+f\left( b\right) }{2}+2f\left( \frac{%
a+b}{2}\right) \right] -\frac{1}{b-a}\dint\limits_{a}^{b}f(x)dx\right\vert
\leq \frac{1}{2880}\left\Vert f^{(4)}\right\Vert _{\infty }\left( b-a\right)
^{4}.
\end{equation*}
\end{theorem}

In 1985, E. K. Godunova and V. I. Levin introduced the following class of
functions (see \cite{GL}):

\begin{definition}
\label{def 1.1} A map $f:I\rightarrow 
\mathbb{R}
$ is said to belong to the class $Q(I)$ if it is nonnegative and for all $%
x,y\in I$ and $\lambda \in \left( 0,1\right) ,$ satisfies the inequality%
\begin{equation*}
f(\lambda x+(1-\lambda )y)\leq \frac{f(x)}{\lambda }+\frac{f(y)}{1-\lambda }.
\end{equation*}
\end{definition}

In \cite{MK}, Moslehian and Kian obtained Hermite-Hadamard and Ostrowski
type inequalities for $Q(I)$ class functions.

In \cite{s}, Sarikaya and Aktan proved the following Lemma:

\begin{lemma}
\label{lem 1.1} Let $I\subset 
\mathbb{R}
$ be an open interval, $a,b\in I$ with $a<b.$ If $f:I\rightarrow 
\mathbb{R}
$ is a twice differentiable mapping such that $f^{\prime \prime }$ is
integrable and $0\leq \lambda \leq 1.$ Then the following identity holds:%
\begin{eqnarray*}
&&\left( \lambda -1\right) f\left( \frac{a+b}{2}\right) -\lambda \frac{%
f(a)+f(b)}{2}+\frac{1}{b-a}\int_{a}^{b}f(x)dx \\
&=&\left( b-a\right) ^{2}\int_{0}^{1}k(t)f^{\prime \prime }\left(
ta+(1-t\right) b)dt
\end{eqnarray*}%
where%
\begin{equation*}
k(t)=\left\{ 
\begin{array}{c}
\frac{1}{2}t\left( t-\lambda \right) ,\text{ \ \ \ \ \ \ \ \ \ \ \ \ \ \ \ \
\ \ }0\leq t\leq \frac{1}{2} \\ 
\\ 
\frac{1}{2}\left( 1-t\right) \left( 1-\lambda -t\right) ,\text{ \ \ \ \ }%
\frac{1}{2}\leq t\leq 1.%
\end{array}%
\right.
\end{equation*}
\end{lemma}

In this paper, using the above Lemma we obtain some new inequalities for
functions which are introduced by Godunova and Levin.

\section{Main Results}

We obtain the following general integral inequalities via Lemma \ref{lem 1.1}%
.

\begin{theorem}
\label{teo 2.1} Let $I\subset 
\mathbb{R}
$ be an open interval, $a,b\in I$ with $a<b$ and $f:I\rightarrow 
\mathbb{R}
$ be a twice differentiable mapping such that $f^{\prime \prime }$ is
integrable. If $\left\vert f^{\prime \prime }\right\vert ^{q}$ belongs to $%
Q(I),$ then the following inequalities hold:%
\begin{eqnarray}
&&  \label{2.1} \\
&&\left\vert \left( \lambda -1\right) f\left( \frac{a+b}{2}\right) -\lambda 
\frac{f(a)+f(b)}{2}+\frac{1}{b-a}\int_{a}^{b}f(x)dx\right\vert  \notag \\
&\leq &\left\{ 
\begin{array}{c}
\frac{\left( b-a\right) ^{2}}{2}\left( \frac{\lambda ^{3}}{3}+\frac{%
1-3\lambda }{24}\right) ^{1-\frac{1}{q}} \\ 
\\ 
\times \left\{ \left( \left[ \lambda ^{2}-\frac{4\lambda -1}{8}\right]
\left\vert f^{\prime \prime }(a)\right\vert ^{q}+\left[ \ln \left( 2\left(
1-\lambda \right) ^{2}\right) ^{1-\lambda }+\frac{20\lambda -8\lambda ^{2}-5%
}{8}\right] \left\vert f^{\prime \prime }(b)\right\vert ^{q}\right) ^{\frac{1%
}{q}}\right. \\ 
\\ 
+\left. \left( \left[ \ln \left( 2\left( 1-\lambda \right) ^{2}\right)
^{1-\lambda }+\frac{20\lambda -8\lambda ^{2}-5}{8}\right] \left\vert
f^{\prime \prime }(a)\right\vert ^{q}+\left[ \lambda ^{2}-\frac{4\lambda -1}{%
8}\right] \left\vert f^{\prime \prime }(b)\right\vert ^{q}\right) ^{\frac{1}{%
q}}\right\} ,\text{ \ \ \ \ for }0\leq \lambda \leq \frac{1}{2} \\ 
\\ 
\frac{\left( b-a\right) ^{2}}{2}\left( \frac{3\lambda -1}{24}\right) ^{^{1-%
\frac{1}{q}}} \\ 
\\ 
\times \left\{ \left( \frac{4\lambda -1}{8}\left\vert f^{\prime \prime
}(a)\right\vert ^{q}+\left[ \frac{5-4\lambda }{8}-\left( \lambda -1\right)
\ln \frac{1}{2}\right] \left\vert f^{\prime \prime }(b)\right\vert
^{q}\right) ^{\frac{1}{q}}\right. \\ 
\\ 
+\left. \left( \left[ \frac{5-4\lambda }{8}-\left( \lambda -1\right) \ln 
\frac{1}{2}\right] \left\vert f^{\prime \prime }(a)\right\vert ^{q}+\frac{%
4\lambda -1}{8}\left\vert f^{\prime \prime }(b)\right\vert ^{q}\right) ^{%
\frac{1}{q}}\right\} ,\text{ \ \ \ \ for }\frac{1}{2}\leq \lambda \leq 1. \\ 
\end{array}%
\right.  \notag
\end{eqnarray}%
where $0\leq \lambda \leq 1$ and $q\geq 1.$
\end{theorem}

\begin{proof}
From Lemma \ref{lem 1.1} and using the power mean inequality, we have%
\begin{eqnarray*}
&&\left\vert \left( \lambda -1\right) f\left( \frac{a+b}{2}\right) -\lambda 
\frac{f(a)+f(b)}{2}+\frac{1}{b-a}\int_{a}^{b}f(x)dx\right\vert \\
&\leq &\left( b-a\right) ^{2}\int_{0}^{1}\left\vert k(t)\right\vert
\left\vert f^{\prime \prime }\left( ta+(1-t\right) b)\right\vert dt \\
&\leq &\frac{\left( b-a\right) ^{2}}{2}\left\{ \int_{0}^{\frac{1}{2}%
}\left\vert t\left( t-\lambda \right) \right\vert \left\vert f^{\prime
\prime }\left( ta+(1-t\right) b)\right\vert dt+\int_{\frac{1}{2}%
}^{1}\left\vert \left( 1-t\right) \left( 1-\lambda -t\right) \right\vert
\left\vert f^{\prime \prime }\left( ta+(1-t\right) b)\right\vert dt\right\}
\\
&\leq &\frac{\left( b-a\right) ^{2}}{2}\left\{ \left( \int_{0}^{\frac{1}{2}%
}\left\vert t\left( t-\lambda \right) \right\vert dt\right) ^{1-\frac{1}{q}%
}\left( \int_{0}^{\frac{1}{2}}\left\vert t\left( t-\lambda \right)
\right\vert \left\vert f^{\prime \prime }\left( ta+(1-t\right) b)\right\vert
^{q}dt\right) ^{\frac{1}{q}}\right. \\
&&\left. +\left( \int_{\frac{1}{2}}^{1}\left\vert \left( 1-t\right) \left(
1-\lambda -t\right) \right\vert dt\right) ^{1-\frac{1}{q}}\left( \int_{\frac{%
1}{2}}^{1}\left\vert \left( 1-t\right) \left( 1-\lambda -t\right)
\right\vert \left\vert f^{\prime \prime }\left( ta+(1-t\right) b)\right\vert
dt\right) ^{\frac{1}{q}}\right\} .
\end{eqnarray*}%
Let $0\leq \lambda \leq \frac{1}{2}.$ Then, since $\left\vert f^{\prime
\prime }\right\vert ^{q}$ belongs to $Q(I)$, we can write for $t\in (0,1)$%
\begin{equation*}
\left\vert f^{\prime \prime }\left( ta+(1-t\right) b)\right\vert ^{q}\leq 
\frac{\left\vert f^{\prime \prime }(a)\right\vert ^{q}}{t}+\frac{\left\vert
f^{\prime \prime }(b)\right\vert ^{q}}{1-t}.
\end{equation*}%
Hence,%
\begin{eqnarray}
&&  \label{2.2} \\
&&\int_{0}^{\frac{1}{2}}\left\vert t\left( t-\lambda \right) \right\vert
\left\vert f^{\prime \prime }\left( ta+(1-t\right) b)\right\vert dt  \notag
\\
&\leq &\int_{0}^{\lambda }t(\lambda -t)\left[ \frac{\left\vert f^{\prime
\prime }(a)\right\vert ^{q}}{t}+\frac{\left\vert f^{\prime \prime
}(b)\right\vert ^{q}}{1-t}\right] dt+\int_{\lambda }^{\frac{1}{2}}t\left(
t-\lambda \right) \left[ \frac{\left\vert f^{\prime \prime }(a)\right\vert
^{q}}{t}+\frac{\left\vert f^{\prime \prime }(b)\right\vert ^{q}}{1-t}\right]
dt  \notag \\
&=&\left[ \lambda ^{2}-\frac{4\lambda -1}{8}\right] \left\vert f^{\prime
\prime }(a)\right\vert ^{q}+\left[ \ln \left( 2\left( 1-\lambda \right)
^{2}\right) ^{1-\lambda }+\frac{20\lambda -8\lambda ^{2}-5}{8}\right]
\left\vert f^{\prime \prime }(b)\right\vert ^{q},  \notag
\end{eqnarray}%
\begin{eqnarray}
&&  \label{2.3} \\
&&\int_{\frac{1}{2}}^{1}\left\vert \left( 1-t\right) \left( 1-\lambda
-t\right) \right\vert \left\vert f^{\prime \prime }\left( ta+(1-t\right)
b)\right\vert dt  \notag \\
&\leq &\int_{\frac{1}{2}}^{1-\lambda }\left( 1-t\right) \left( 1-\lambda
-t\right) \left[ \frac{\left\vert f^{\prime \prime }(a)\right\vert ^{q}}{t}+%
\frac{\left\vert f^{\prime \prime }(b)\right\vert ^{q}}{1-t}\right] dt 
\notag \\
&&+\int_{1-\lambda }^{1}\left( 1-t\right) \left( t+\lambda -1\right) \left[ 
\frac{\left\vert f^{\prime \prime }(a)\right\vert ^{q}}{t}+\frac{\left\vert
f^{\prime \prime }(b)\right\vert ^{q}}{1-t}\right] dt  \notag \\
&=&\left[ \ln \left( 2\left( 1-\lambda \right) ^{2}\right) ^{1-\lambda }+%
\frac{20\lambda -8\lambda ^{2}-5}{8}\right] \left\vert f^{\prime \prime
}(a)\right\vert ^{q}+\left[ \lambda ^{2}-\frac{4\lambda -1}{8}\right]
\left\vert f^{\prime \prime }(b)\right\vert ^{q},  \notag
\end{eqnarray}%
\begin{equation}
\int_{0}^{\frac{1}{2}}\left\vert t\left( t-\lambda \right) \right\vert
dt=\int_{0}^{\lambda }t\left( \lambda -t\right) dt+\int_{\lambda }^{\frac{1}{%
2}}t\left( t-\lambda \right) dt=\frac{\lambda ^{3}}{3}+\frac{1-3\lambda }{24}
\label{2.4}
\end{equation}%
and%
\begin{equation}
\int_{\frac{1}{2}}^{1}\left\vert \left( 1-t\right) \left( 1-\lambda
-t\right) \right\vert dt=\int_{\frac{1}{2}}^{1-\lambda }\left( 1-t\right)
\left( 1-\lambda -t\right) dt+\int_{1-\lambda }^{1}\left( 1-t\right) \left(
t+\lambda -1\right) dt=\frac{\lambda ^{3}}{3}+\frac{1-3\lambda }{24}.
\label{2.5}
\end{equation}%
If we use (\ref{2.2})-(\ref{2.5}) in (\ref{2.1}), we obtain the first
inequality in (\ref{2.1}).

Let $\frac{1}{2}\leq \lambda \leq 1.$ Then,%
\begin{eqnarray}
&&  \label{2.6} \\
&&\int_{0}^{\frac{1}{2}}\left\vert t\left( t-\lambda \right) \right\vert
\left\vert f^{\prime \prime }\left( ta+(1-t\right) b)\right\vert dt  \notag
\\
&\leq &\int_{0}^{\frac{1}{2}}t\left( \lambda -t\right) \left[ \frac{%
\left\vert f^{\prime \prime }(a)\right\vert ^{q}}{t}+\frac{\left\vert
f^{\prime \prime }(b)\right\vert ^{q}}{1-t}\right] dt  \notag \\
&=&\frac{4\lambda -1}{8}\left\vert f^{\prime \prime }(a)\right\vert ^{q}+%
\left[ \frac{5-4\lambda }{8}-\left( \lambda -1\right) \ln \frac{1}{2}\right]
\left\vert f^{\prime \prime }(b)\right\vert ^{q},  \notag
\end{eqnarray}%
\begin{eqnarray}
&&  \label{2.7} \\
&&\int_{\frac{1}{2}}^{1}\left\vert \left( 1-t\right) \left( 1-\lambda
-t\right) \right\vert \left\vert f^{\prime \prime }\left( ta+(1-t\right)
b)\right\vert dt  \notag \\
&\leq &\int_{\frac{1}{2}}^{1}\left( 1-t\right) \left( t+\lambda -1\right) %
\left[ \frac{\left\vert f^{\prime \prime }(a)\right\vert ^{q}}{t}+\frac{%
\left\vert f^{\prime \prime }(b)\right\vert ^{q}}{1-t}\right] dt  \notag \\
&=&\left[ \frac{5-4\lambda }{8}-\left( \lambda -1\right) \ln \frac{1}{2}%
\right] \left\vert f^{\prime \prime }(a)\right\vert ^{q}+\frac{4\lambda -1}{8%
}\left\vert f^{\prime \prime }(b)\right\vert ^{q},  \notag
\end{eqnarray}%
and%
\begin{equation}
\int_{0}^{\frac{1}{2}}\left\vert t\left( t-\lambda \right) \right\vert
dt=\int_{\frac{1}{2}}^{1}\left\vert \left( 1-t\right) \left( 1-\lambda
-t\right) \right\vert dt=\frac{3\lambda -1}{24}.  \label{2.8}
\end{equation}%
If we use (\ref{2.6})-(\ref{2.8}) in (\ref{2.1}), we obtain the second
inequality in (\ref{2.1}). The proof is completed.
\end{proof}

\begin{corollary}
\label{co 1.1} In Theorem \ref{teo 2.1}, if we choose $q=1$ we obtain the
following inequalities:%
\begin{eqnarray*}
&&\left\vert \left( \lambda -1\right) f\left( \frac{a+b}{2}\right) -\lambda 
\frac{f(a)+f(b)}{2}+\frac{1}{b-a}\int_{a}^{b}f(x)dx\right\vert \\
&\leq &\left\{ 
\begin{array}{c}
\frac{\left( b-a\right) ^{2}}{2}\left[ \ln \left( 2\left( 1-\lambda \right)
^{2}\right) ^{1-\lambda }+\frac{16\lambda -4}{8}\right] \left[ \left\vert
f^{\prime \prime }(a)\right\vert +\left\vert f^{\prime \prime
}(b)\right\vert \right] ,\text{ \ \ \ \ for }0\leq \lambda \leq \frac{1}{2}
\\ 
\\ 
\frac{\left( b-a\right) ^{2}}{2}\left[ \ln 2^{\lambda -1}e^{\frac{1}{2}}%
\right] \left[ \left\vert f^{\prime \prime }(a)\right\vert +\left\vert
f^{\prime \prime }(b)\right\vert \right] ,\text{ \ \ \ \ \ \ \ \ \ \ \ \ \ \
\ \ \ \ \ \ \ \ \ \ \ \ \ \ \ \ for }\frac{1}{2}\leq \lambda \leq 1.%
\end{array}%
\right.
\end{eqnarray*}
\end{corollary}

\section{Applications}

\begin{proposition}
\label{prop 3.1} If we choose $\lambda =0$ in Corollary \ref{co 1.1}, we
obtain the following inequality:%
\begin{equation*}
\left\vert \frac{1}{b-a}\int_{a}^{b}f(x)dx-f\left( \frac{a+b}{2}\right)
\right\vert \leq \frac{\left( b-a\right) ^{2}}{4}\left( \ln \frac{4}{e}%
\right) \left[ \left\vert f^{\prime \prime }(a)\right\vert +\left\vert
f^{\prime \prime }(b)\right\vert \right] .
\end{equation*}
\end{proposition}

\begin{proposition}
\label{prop 3.2} If we choose $\lambda =1$ in Corollary \ref{co 1.1}, we
obtain the following inequality:%
\begin{equation*}
\left\vert \frac{1}{b-a}\int_{a}^{b}f(x)dx-\frac{f(a)+f(b)}{2}\right\vert
\leq \frac{\left( b-a\right) ^{2}}{2}\left( \ln e^{\frac{1}{2}}\right) \left[
\left\vert f^{\prime \prime }(a)\right\vert +\left\vert f^{\prime \prime
}(b)\right\vert \right] .
\end{equation*}
\end{proposition}

\begin{proposition}
\label{prop 3.3} If we choose $\lambda =\frac{1}{3}$ in Corollary \ref{co
1.1}, we obtain the following inequality:%
\begin{equation*}
\left\vert \frac{1}{6}\left[ f(a)+4f\left( \frac{a+b}{2}\right) +f(b)\right]
-\frac{1}{b-a}\int_{a}^{b}f(x)dx\right\vert \leq \frac{\left( b-a\right) ^{2}%
}{2}\left[ \frac{2}{3}\ln \frac{8}{9}+\frac{1}{6}\right] \left[ \left\vert
f^{\prime \prime }(a)\right\vert +\left\vert f^{\prime \prime
}(b)\right\vert \right] .
\end{equation*}
\end{proposition}

\begin{proposition}
\label{prop 3.4} If we choose $\lambda =\frac{1}{2}$ in Corollary \ref{co
1.1}, we obtain the following inequality:%
\begin{equation*}
\left\vert \frac{1}{b-a}\int_{a}^{b}f(x)dx-\frac{1}{2}\left[ f\left( \frac{%
a+b}{2}\right) +\frac{f(a)+f(b)}{2}\right] \right\vert \leq \frac{\left(
b-a\right) ^{2}}{4}\left( \ln \frac{1}{2}+1\right) \left[ \left\vert
f^{\prime \prime }(a)\right\vert +\left\vert f^{\prime \prime
}(b)\right\vert \right] .
\end{equation*}
\end{proposition}

\begin{proposition}
\label{prop 3.5} If we choose $\lambda =0$ in Theorem \ref{teo 2.1}, we
obtain the following inequality:%
\begin{eqnarray*}
\left\vert \frac{1}{b-a}\int_{a}^{b}f(x)dx-f\left( \frac{a+b}{2}\right)
\right\vert &\leq &\frac{\left( b-a\right) ^{2}}{2}\left( \frac{1}{24}%
\right) ^{1-\frac{1}{q}}\left\{ \left( \frac{1}{8}\left\vert f^{\prime
\prime }(a)\right\vert ^{q}+\left( \ln 2-\frac{5}{8}\right) \left\vert
f^{\prime \prime }(b)\right\vert ^{q}\right) ^{\frac{1}{q}}\right. \\
&&\left. +\left( \left( \ln 2-\frac{5}{8}\right) \left\vert f^{\prime \prime
}(a)\right\vert ^{q}+\frac{1}{8}\left\vert f^{\prime \prime }(b)\right\vert
^{q}\right) ^{\frac{1}{q}}\right\} .
\end{eqnarray*}
\end{proposition}

\begin{proposition}
\label{prop 3.6} If we choose $\lambda =1$ in Theorem \ref{teo 2.1}, we
obtain the following inequality:%
\begin{eqnarray*}
\left\vert \frac{1}{b-a}\int_{a}^{b}f(x)dx-\frac{f(a)+f(b)}{2}\right\vert
&\leq &\frac{\left( b-a\right) ^{2}}{2}\left( \frac{1}{12}\right) ^{1-\frac{1%
}{q}}\left\{ \left( \frac{3}{8}\left\vert f^{\prime \prime }(a)\right\vert
^{q}+\frac{1}{8}\left\vert f^{\prime \prime }(b)\right\vert ^{q}\right) ^{%
\frac{1}{q}}\right. \\
&&\left. +\left( \frac{1}{8}\left\vert f^{\prime \prime }(a)\right\vert ^{q}+%
\frac{3}{8}\left\vert f^{\prime \prime }(b)\right\vert ^{q}\right) ^{\frac{1%
}{q}}\right\} .
\end{eqnarray*}
\end{proposition}

\begin{proposition}
\label{prop 3.7} If we choose $\lambda =\frac{1}{3}$ in Theorem \ref{teo 2.1}%
, we obtain the following inequality:%
\begin{eqnarray*}
&&\left\vert \frac{1}{6}\left[ f(a)+4f\left( \frac{a+b}{2}\right) +f(b)%
\right] -\frac{1}{b-a}\int_{a}^{b}f(x)dx\right\vert \\
&\leq &\frac{\left( b-a\right) ^{2}}{2}\left( \frac{1}{81}\right) ^{1-\frac{1%
}{q}}\left\{ \left( \frac{5}{72}\left\vert f^{\prime \prime }(a)\right\vert
^{q}+\left( \frac{2}{3}\ln \frac{8}{9}+\frac{7}{72}\right) \left\vert
f^{\prime \prime }(b)\right\vert ^{q}\right) ^{\frac{1}{q}}\right. \\
&&\left. \left( \left( \frac{2}{3}\ln \frac{8}{9}+\frac{7}{72}\right)
\left\vert f^{\prime \prime }(a)\right\vert ^{q}+\frac{5}{72}\left\vert
f^{\prime \prime }(b)\right\vert ^{q}\right) ^{\frac{1}{q}}\right\} .
\end{eqnarray*}
\end{proposition}

\end{document}